\documentclass[11pt]{article}

\usepackage{amssymb}
\usepackage{amsmath}
\usepackage{color}
\usepackage{theorem}
\usepackage[all]{xy}
\usepackage{pgf,acadpgf}

\newcommand{\XYMATRIX}{\xymatrix@M=6pt}
\newcommand{\aremb}{\ar@{^{(}->}}
\newcommand{\arembfrom}{\ar@{<-^{)}}}

\numberwithin{equation}{section}

\theorembodyfont{\slshape}

  \newtheorem{THM}{Theorem}[section]

  \newtheorem{COR}[THM]{Corollary}

\theorembodyfont{\rmfamily}

\newif\ifQEDsign
\newcommand{\QED}{\global\QEDsigntrue\hfill$\square$}

\newenvironment{proof}%
    {\par\noindent\textit{Proof.}\global\QEDsignfalse}%
    {\ifQEDsign\else\QED\fi\par\bigskip\par}

\renewcommand{\preceq}{\preccurlyeq}

\newcommand{\lex}{\mathrel{<_{\mathit{lex}}}}
\newcommand{\alex}{\mathrel{<_{\mathit{alex}}}}

\newcommand{\clex}{\mathrel{<_{\overline{\mathit{lex}}}}}
\renewcommand{\le}{\leqslant}
\renewcommand{\ge}{\geqslant}

\newcommand{\0}{\varnothing}

\renewcommand{\phi}{\varphi}
\renewcommand{\epsilon}{\varepsilon}
\newcommand{\UNION}{\bigcup}

\newcommand{\DD}{\mathbf{D}}

\newcommand{\KK}{\mathbf{K}}
\newcommand{\LL}{\mathbf{L}}

\newcommand{\NN}{\mathbb{N}}
\newcommand{\PP}{\mathbf{P}}

\newcommand{\SSS}{\mathbf{S}}

\newcommand{\VV}{\mathbf{V}}
\newcommand{\WW}{\mathbf{W}}
\newcommand{\XX}{\mathbf{X}}

\newcommand{\union}{\cup}
\newcommand{\restr}[2]{\hbox{$#1$}\hbox{$\upharpoonright$}_{#2}}

\newcommand{\Boxed}[1]{\mbox{$#1$}}
\newcommand{\embedsto}{\hookrightarrow}
\newcommand{\id}{\mathrm{id}}

\newcommand{\LO}{\mathrm{LO}}
\newcommand{\Sym}{\mathrm{Sym}}

\newcommand{\fin}{\mathit{fin}}

\newcommand{\im}{\mathrm{im}}
\newcommand{\Age}{\mathrm{Age}}

\newcommand{\OPP}{\overrightarrow{\PP}}

\newcommand{\calA}{\mathcal{A}}
\newcommand{\calB}{\mathcal{B}}
\newcommand{\calC}{\mathcal{C}}
\newcommand{\calD}{\mathcal{D}}
\newcommand{\calE}{\mathcal{E}}
\newcommand{\calF}{\mathcal{F}}

\newcommand{\calL}{\mathcal{L}}

\newcommand{\calP}{\mathcal{P}}
\newcommand{\calQ}{\mathcal{Q}}

\newcommand{\calS}{\mathcal{S}}
\newcommand{\calT}{\mathcal{T}}

\newcommand{\calX}{\mathcal{X}}

\newcommand{\Fraisse}{Fra\"\i ss\'e}

\DeclareMathOperator{\Aut}{Aut}

\DeclareMathOperator{\rel}{rel}

\title{The Ramsey and the ordering property for classes of lattices and semilattices}
\author{%
  Dragan Ma\v sulovi\'c\\
  University of Novi Sad, Faculty of Sciences\\
  Department of Mathematics and Informatics\\
  Trg Dositeja Obradovi\'ca 3, 21000 Novi Sad, Serbia\\
  e-mail: dragan.masulovic@dmi.uns.ac.rs}

\begin{document}
\maketitle

\begin{abstract}
  The class of finite distributive lattices, as many other classes of structures,
  does not have the Ramsey property. It is quite common,
  though, that after expanding the structures with appropriately chosen linear orders
  the resulting class has the Ramsey property. So, one might expect
  that a similar result holds for the class of all finite distributive lattices.
  Surprisingly, Kechris and Soki\'c have proved in 2012 that this is not the case:
  no expansion of the class of finite distributive lattices by linear orders satisfies the
  Ramsey property.

  In this paper we prove that the variety of distributive lattices is not an exception, but an instance of a more general phenomenon.
  We show that for almost all nontrivial locally finite varieties of lattices no ``reasonable'' expansion of the finite members
  of the variety by linear orders gives rise to a Ramsey class. The responsibility for this
  lies not with the lattices as structures, but with the lack of algebraic morphisms:
  if we consider lattices as partially ordered sets (and thus switch from algebraic embeddings
  to embeddings of relational structures) we show that every variety of lattices
  gives rise to a class of linearly ordered posets having both the Ramsey property and
  the ordering property. It now comes as no surprise that the same is true for varieties of semilattices.
  
  \bigskip


  \bigskip

  \noindent \textbf{AMS Subj.\ Classification (2010):} 05C55, 06B20
\end{abstract}

\section{Introduction}

Generalizing the classical results of F.~P.~Ramsey from the late 1920's, the structural Ramsey theory originated at
the beginning of 1970’s in a series of papers (see \cite{N1995} for references).
We say that a class $\KK$ of finite structures has the \emph{Ramsey property} if the following holds:
for any number $k \ge 2$ of colors and all $\calA, \calB \in \KK$ such that $\calA$ embeds into $\calB$
there is a $\calC \in \KK$
such that no matter how we color the copies of $\calA$ in $\calC$ with $k$ colors, there is a \emph{monochromatic} copy
$\calB'$ of $\calB$ in $\calC$ (that is, all the copies of $\calA$ that fall within $\calB'$ are colored by the same color).

Many natural classes of structures (such as finite graphs, metric spaces and partially ordered sets, just to
name a few) do not have the Ramsey property, and lattices as algebras with two binary operations
satisfying certain algebraic laws are not an exception:
the class of all finite lattices, the class of all finite distributive lattices and the class of all finite
modular lattices do not have the Ramsey property (see~\cite{Nesetril-Rodl-lattices,Promel-Voigt}).
This is not surprising as all these classes contain non-rigid structures, and it has been established
relatively recently that a necessary condition for a class of finite structures to have the Ramsey property is that all its elements
be rigid (that is, have trivial automorphism groups)~\cite{Nesetril,masul-scow}.

It is quite common, though, that after expanding the structures under consideration
with appropriately chosen linear orders, the resulting class of expanded
structures has the Ramsey property. For example, the class of all finite
linearly ordered graphs $(V, E, \Boxed<)$ where $(V, E)$ is a finite graph and $<$ is
a linear order on the set $V$ of vertices of the graph has the Ramsey property~\cite{AH,Nesetril-Rodl}.
The same is true for metric spaces~\cite{Nesetril-metric}.
In case of finite posets we consider the class of all
finite linearly ordered posets $(P, \Boxed\preceq, \Boxed<)$ where $(P, \Boxed\preceq)$ is
a finite poset and $<$ is a linear order on $P$ which extends~$\preceq$~\cite{PTW,fouche}.
Moreover, in~\cite{sokic-semilat} several classes of semilattices
have been shown to have the Ramsey property if the semilattices in the class
are expanded by appropriate linear orders.
So, one might expect that a similar result holds for finite lattices.
Surprisingly, this is not the case. In~\cite{kechris-sokic} the authors prove that
no expansion of the class of finite distributive lattices by linear orders satisfies the
Ramsey property.

We start Section~\ref{wtc.sec.slat} by showing that
the variety of distributive lattices is not an exception, but an instance of a more general phenomenon.
We show that for an arbitrary nontrivial locally finite variety $\VV$
of lattices distinct from the variety of all the lattices and the variety of distributive lattices,
no ``reasonable'' expansion of $\VV^\fin$ ($=$ the class of all the finite lattices in $\VV$) by linear orders
has the Ramsey property. So, it seems that lattices are simply not fit for the Ramsey property.
Our main goal in Section~\ref{wtc.sec.slat} is to demonstrate that the responsibility for this
lies not with the lattices as structures, but with the lack of algebraic morphisms:
if we consider lattices as partially ordered sets (and thus switch from algebraic embeddings
to embeddings of relational structures) we show that every variety of lattices
gives rise to a class of linearly ordered posets having both the Ramsey property and
the ordering property (a property related to the Ramsey property which we define in Section~\ref{wtc.sec.prelim}
along with other notions we use in the paper).
Namely, there are much more embeddings between two lattices understood as relational structures than
there are embeddings between the same two lattices understood as algebras, and
this abundance of embeddings between relational structures is the key reason we are able to prove that
all the varieties of lattices have the Ramsey property in their ``relational alter ego''.
It now comes as no surprise that the same is true for varieties of semilattices.

In many particular cases the ordering property is implied by nontrivial Ramsey properties~\cite{Nesetril}, and
one such particular case is demonstrated in Section~\ref{wtc.sec.slat}.
Using the standard Sierpinski-style coloring obtained by comparing two linear orders
we can derive the ordering property for a class of linearly ordered posets from the fact that it has the Ramsey property.
In Section~\ref{wtc.sec.wtc} we generalize this idea to arbitrary classes of
first-order structures which satisfy a model-theoretic requirement we refer to as the \emph{weak triangle condition}.
This is a weaker form of the \emph{triangle condition} introduced in~\cite{KPT} in connection to understanding the consequences of the
ordering property. As an example we apply the main result of Section~\ref{wtc.sec.wtc} to show the ordering property for a class of
finite structures consisting of a set together with several partially ordered sets that conform to a given template.
This example is instructive since we do not see an easy way to derive the ordering property for this class directly.

We conclude the paper by a discussion of the importance of the Ramsey and the ordering properties
in the context of the Kechris-Pestov-Todor\v cevi\'c correspondence~\cite{KPT}, an intricate interplay of
discrete mathematics, model theory and topological dynamics. As the final example
we present a new infinite family of topological groups whose universal minimal flows can be computed using
this correspondence.

\section{Preliminaries}
\label{wtc.sec.prelim}

\paragraph{First-order structures.}
Let $\Theta = \Theta_R \union \Theta_F$ be a first-order language 
where $\Theta_R$ is a set of finitary relational symbols, and $\Theta_F$ is a set of finitary functional symbols.
Whereas we do not allow relational symbols of arity~0, functional symbols of arity~0 are welcome and correspond to
constants. A \emph{$\Theta$-structure} $\calA = (A, \Theta^\calA)$
is a set $A$ together with a set $\Theta^\calA$ of finitary relations on $A$ and finitary functions on $A$
which are the interpretations of the corresponding symbols in $\Theta$.
A \emph{relational language} is a first order language $\Theta$ where $\Theta_F = \0$.
An \emph{algebraic language} is a first order language $\Theta$ where $\Theta_R = \0$.
If $\Theta$ is a relational language, $\Theta$-structures are then usually referred to as
\emph{$\Theta$-relational structures}; and if $\Theta$ is an algebraic language,
$\Theta$-structures are then usually referred to as \emph{$\Theta$-algebras}.

Structures will be denoted by script letters $\calA$, $\calB_1$, $\calC^*$, \ldots, and
the underlying set of a structure $\calA$, $\calB_1$, $\calC^*$, \ldots\ will always be denoted by its roman
letter $A$, $B_1$, $C^*$, \ldots\ respectively. A structure $\calA$ is \emph{finite (countably infinite)}
if $A$ is a finite (countably infinite) set.
For a class $\KK$ of structures, by $\KK^\fin$ we denote the class of all the finite structures in~$\KK$.

An \emph{embedding} $f: \calA \hookrightarrow \calB$ between two $\Theta$-structures
is every injective map $f: A \rightarrow B$ satisfying the following:
\begin{itemize}
\item
  for every $\theta \in \Theta_R$ we have that
  $(a_1, \ldots, a_r) \in \theta^\calA \Leftrightarrow (f(a_1), \ldots, f(a_r)) \in \theta^\calB$,
  where $r$ is the arity of~$\theta$; and
\item
  for every $\phi \in \Theta_F$ we have that
  $f(\phi^\calA(a_1, \ldots, a_r)) = \phi^\calB(f(a_1), \ldots, f(a_r))$,
  where $r$ is the arity of~$\phi$.
\end{itemize}
Surjective embeddings are \emph{isomorphisms}. Structures $\calA$ and $\calB$ are \emph{isomorphic},
and we write $\calA \cong \calB$, if there is an isomorphism $\calA \to \calB$.
An \emph{automorphism} is an isomorphism $\calA \to \calA$. By $\Aut(\calA)$ we denote the
set of all the automorphisms of a structure~$\calA$. A structure $\calA$ is \emph{rigid}
if $\Aut(\calA) = \{\id_A\}$.

A structure $\calA$ is a \emph{substructure} of a structure $\calB$,
and we write $\calA \le \calB$, if $A \subseteq B$ and the identity map $a \mapsto a$ is an embedding of $\calA$ into $\calB$.
A \emph{substructure of a structure $\calA$ generated by $S \subseteq A$} is the least (with respect to inclusion)
substructure $\calB$ of $\calA$ such that $S \subseteq B$. We denote by $\langle S \rangle_\calA$
the substructure of $\calA$ generated by $S \subseteq A$.
A structure $\calA$ is \emph{locally finite} if every finitely generated substructure of $\calA$ is finite.

A poset $\calE = (E, \Boxed \sqsubseteq)$ is a \emph{lattice}
if every pair of elements of $E$ has the greatest lower bound and the least upper bound.
Alternatively, an algebra $\calL = (L, \Boxed\land, \Boxed\lor)$ with two binary operations is a \emph{lattice}
if both operations are idempotent, commutative and associative, and the absorptive laws hold.
These two points of view are closely related: every poset
$(E, \Boxed\sqsubseteq)$ which is a lattice uniquely determines operations
$\Boxed{\land}, \Boxed{\lor} : E^2 \to E$
such that the algebra $(E, \Boxed{\land}, \Boxed{\lor})$ is a lattice
(take $a \land b$, resp.\ $a \lor b$, to be the greatest lower bound, resp.\ the least upper bound, for $a$ and $b$).
Conversely, every algebra $(L, \Boxed\land, \Boxed\lor)$ which is a lattice uniquely determines the partial order
$\Boxed{\sqsubseteq} \subseteq L^2$ such that the poset $(L, \Boxed{\sqsubseteq})$ is a lattice
(take $a \sqsubseteq b$ if and only if $a \land b = a$).
What makes these two approaches to lattices different are the embeddings.
Clearly, every embedding $f : (L_1, \Boxed{\land_1}, \Boxed{\lor_1}) \to (L_2, \Boxed{\land_2}, \Boxed{\lor_2})$
between algebras is also an embedding $(L_1, \Boxed{\sqsubseteq_1}) \to (L_2, \Boxed{\sqsubseteq_2})$, where $\Boxed{\sqsubseteq_1}$
and $\Boxed{\sqsubseteq_2}$ are the corresponding derived lattice-ordering relations. The converse, however, is not true.

A poset $\calE = (E, \Boxed\sqsubseteq)$ is a \emph{(meet) semilattice} if every pair of elements of $E$ has the
greatest lower bound. An algebra $\calS = (L, \Boxed\land)$ with one binary operation is a \emph{(meet) semilattice}
if the operation is idempotent, commutative and associative.

\paragraph{Classes of structures.}
A \emph{variety} of algebras is a class of algebras over a fixed algebraic language which is
closed with respect to taking homomorphic images, subalgebras and products of arbitrary families of algebras from the class.
The study of varieties of lattices is a deep and active research field in modern algebra, and we refer the reader to
\cite{Jipsen-Rose} for more insight into the typical problems addressed in this context.
Clearly, the class $\SSS$ of all the semilattices as algebras (of all cardinalities)
is a variety of semilattices and the class $\LL$ of all the lattices as algebras (of all cardinalities)
is a variety of lattices. Let $\DD$ denote the variety of all the distributive lattices.

\Fraisse\ theory is a deep structural theory of classes of relational structures.
The \emph{age} of a countably infinite structure $\calA$ is the class of all the finite
structures that embed into~$\calA$. The age of $\calA$ will be denoted by~$\Age(\calA)$.
A class $\KK$ of finite structures is an \emph{age} if there is countably infinite structure $\calA$ such that
$\KK = \Age(\calA)$. It is easy to see that a class $\KK$ of finite structures is an age if and only if
$\KK$ is an abstract class (that is, closed for isomorphisms),
there are at most countably many pairwise nonisomorphic structures in $\KK$,
$\KK$ has the \emph{hereditary property}:
\begin{description}
\item[(HP)]
  if $\calA \in \KK$ and $\calB \hookrightarrow \calA$ then $\calB \in \KK$;
\end{description}
and $\KK$ has the \emph{joint embedding property}:
\begin{description}
\item[(JEP)]
  for all $\calA, \calB \in \KK$ there is a $\calC \in \KK$ such that
  $\calA \hookrightarrow \calC$ and $\calB \hookrightarrow \calC$.
\end{description}
An age $\KK$ is a \emph{\Fraisse\ age} (= \Fraisse\ class = amalgamation class)
\cite{Fraisse1,Fraisse2} if $\KK$ satisfies the
\emph{amalgamation property}:
\begin{description}
\item[(AP)]
    for all $\calA, \calB, \calC \in \KK$ and embeddings $f : \calA \hookrightarrow \calB$ and
    $g : \calA \hookrightarrow \calC$ there exist $\calD \in \KK$ and embeddings $f' : \calB \hookrightarrow \calD$ and
    $g' : \calC \hookrightarrow \calD$ such that $f' \circ f = g' \circ g$.
\end{description}

A structure $\calC$ is \emph{ultrahomogeneous} if for every finitely generated structure $\calA$ and every pair of
embeddings $f, g : \calA \hookrightarrow \calC$ there is an automorphism $h \in \Aut(\calC)$ such that
$f = h \circ g$. The age of every locally finite countably infinite ultrahomogeneous structure
is a \Fraisse\ age~\cite{Fraisse1,Fraisse2}.
Conversely, for every \Fraisse\ age $\KK$ there is a unique (up to isomorphism)
countably infinite locally finite ultrahomogeneous structure $\calA$ such that $\KK = \Age(\calA)$~\cite{Fraisse1,Fraisse2}.
We say that $\calA$ is the \emph{\Fraisse\ limit} of $\KK$. For details on \Fraisse\ theory
and further model theoretic background we refer the reader to~\cite{hodges}.

\paragraph{Ramsey theory.}
The leitmotif of Ramsey theory is to prove the existence of regular patterns that occur when a large structure
is considered in a restricted context. It started with the nowadays famous Ramsey theorem whose finite version takes the
following form where for a set $S$ and a positive integer $k$ by $\binom Sk$ we denote the
set of all the $k$-element subsets of~$S$:

\begin{THM}[Finite Ramsey Theorem \cite{Ramsey}]\label{wtc.thm.FRT}
  For positive integers $k$, $m$ and $r$ there exists an integer $n$ such that for every coloring
  $\chi : \binom nk \to \{1, 2, \ldots, r\}$
  there exists a set $S \in \binom nm$ such that $\chi(X) = \chi(Y)$ for all $X, Y \in \binom S k$.
\end{THM}

The Graham-Rothschild Theorem (Theorem~\ref{wtc.thm.G-R} below) is one of the most powerful
tools in Ramsey theory. It formulation requires some preparation.
Let $A$ be a finite alphabet. A word $u$ of length $n$ over $A$ can be thought of as
an element of $A^n$ but also as a mapping $u : \{1, 2, \ldots, n\} \to A$. In the latter case
$u^{-1}(a)$, $a \in A$, denotes the set of all the positions in $u$ where $a$ appears.

Let $X = \{x_1, x_2, \ldots\}$ be a countably infinite set of variables and let
$A$ be a finite alphabet disjoint from $X$. An \emph{$m$-parameter word over $A$ of length $n$} is a word
$w \in (A \union \{x_1, x_2, \ldots, x_m\})^n$ satisfying the following:
\begin{itemize}
\item
  each of the letters $x_1, \ldots, x_m$ appears at least once in $w$, and
\item
  $\min(w^{-1}(x_i)) < \min(w^{-1}(x_j))$ whenever $1 \le i < j \le m$.
\end{itemize}
Let $W^n_m(A)$ denote the set of all the $m$-parameter words over $A$ of length~$n$.
For $u \in W^n_m(A)$ and $v = v_1 v_2 \ldots v_m \in W^m_k(A)$ let
$$
  u \cdot v = u[v_1/x_1, v_2/x_2, \ldots, v_m/x_m] \in W^n_k(A)
$$
denote the word obtained by replacing each occurrence of $x_i$ in $u$ with $v_i$,
simultaneously for all $i \in \{1, \ldots, m\}$.

\begin{THM}[Graham, Rothschild~\cite{GR}]\label{wtc.thm.G-R}
  Let $A$ be a finite alphabet and let
  $m, \ell \ge 1$ and $k \ge 2$. Then there exists an $n$ such that for every partition
  $W^n_\ell(A) = \calX_1 \union \ldots \union \calX_k$ there exist a $u \in W^n_m(A)$ and $j$ such that
  $\{u \cdot v : v \in W^m_\ell(A)\} \subseteq \calX_j$.
\end{THM}

\paragraph{Structural Ramsey theory.}
Generalizing the Finite Ramsey Theorem, the structural Ramsey theory originated at
the beginning of 1970’s in a series of papers (see \cite{N1995} for references).
Let $\Theta$ be a first-order language,
let $\calA$, $\calB$ and $\calC$ be finite $\Theta$-structures and let $k \ge 2$ be an integer.
Let $\binom \calB \calA = \{\tilde \calA : \tilde \calA \le \calB$ and $\tilde \calA \cong \calA\}$.
We write $\calC \longrightarrow (\calB)^\calA_k$
to denote the following: for every partition
$\binom \calC \calA = \calX_1 \union \ldots \union \calX_k$
there exist $\tilde \calB \in \binom \calC \calB$ and $j \in \{1, \ldots, k\}$ such that
$\binom{\tilde \calB}{\calA} \subseteq \calX_j$.
A class $\KK$ of finite $\Theta$-structures
has the \emph{Ramsey property} if the following holds:
\begin{description}
\item[(RP)]
  for all $\calA, \calB \in \KK$ such that $\calA \hookrightarrow \calB$
  and any integer $k \ge 2$ there is a $\calC \in \KK$ such that $\calC \longrightarrow (\calB)^\calA_k$.
\end{description}

A class $\KK$ of finite structures is a \emph{Ramsey age} if it has (HP), (JEP) and (RP).
If a class $\KK$ of finite $\Theta$-structures has (RP) and (JEP) then $\KK$ also has (AP)~\cite{Nesetril}.
So, every Ramsey age is a \Fraisse\ age.

\paragraph{The ordering property.}
The ordering property is a property related to the Ramsey property and in many particular cases
is implied by nontrivial Ramsey properties~\cite{Nesetril}.
The ordering property was introduced in~\cite{Nesetril-Rodl-1975,Nesetril-Rodl-OP-GRA} and has since played an important role in
Structural Ramsey theory.

Let $\Theta$ be a first-order language and let $\Boxed< \notin \Theta$
be a new binary relational symbol. Let $\Theta^* = \Theta \union \{\Boxed<\}$.
Given a $\Theta^*$-structure $\calA$, we shall always interpret 
$<$ in $\calA$ as a linear order on $A$.
A class $\KK^*$ of $\Theta^*$-structures is
an \emph{order expansion} of the class $\KK$ of $\Theta$-structures if
  for every $(\calA, \Boxed<) \in \KK^*$ we have $\calA \in \KK$, and
  for every $\calA \in \KK$ there is at least one linear order $<$ on $A$ such that
  $(\calA, \Boxed<) \in \KK^*$.

For a class $\KK^*$ of $\Theta^*$-structures let $\restr{\KK^*}{\Theta} = \{\calA : (\calA, \Boxed<) \in \KK^*\}$.
Clearly, $\restr{\KK^*}{\Theta}$ is a class of $\Theta$-structures.

An order expansion $\KK^*$ of $\KK$ is \emph{reasonable}~\cite{KPT} if
  for all $\calA, \calB \in \KK$, every embedding $f : \calA \hookrightarrow \calB$ and every
  linear order $<$ on $A$ such that $(\calA, \Boxed<) \in \KK^*$ there is a linear order $\sqsubset$ on $B$
  such that $(\calB, \Boxed\sqsubset) \in \KK^*$ and $f$ is an embedding of $(\calA, \Boxed<)$ into $(\calB, \sqsubset)$.
An order expansion $\KK^*$ of $\KK$ is a \emph{reasonable (JEP)-expansion} of $\KK$ if
$\KK^*$ is a reasonable expansion of $\KK$ and $\KK^*$ has (JEP).

The class $\KK^*$ of finite $\Theta^*$-structures has the \emph{ordering property} if the following holds, where
$\KK = \restr{\KK^*}{\Theta}$:
\begin{description}
\item[(OP)]
  for every $\calA \in \KK$ there is a $\calB \in \KK$ such that $(\calA, \Boxed<) \embedsto (\calB, \Boxed\sqsubset)$
  for every linear order $<$ on $A$ such that $(\calA, \Boxed<) \in \KK^*$, and every linear order
  $\sqsubset$ on $B$ such that $(\calB, \Boxed\sqsubset) \in \KK^*$. We say that $\calB$ is a \emph{witness
  for the ordering property for $\calA$}.
\end{description}

\section{Varieties of lattices and semilattices as classes of relational structures}
\label{wtc.sec.slat}

It was shown in~\cite{kechris-sokic} that no order expansion of $\DD^\fin$ has the
Ramsey property. We shall now show that this is an ubiquitous phenomenon when it comes to locally
finite varieties of lattices.

\begin{THM}\label{wtc.thm.norp}
  Let $\VV$ be a nontrivial locally finite variety of lattices distinct from $\LL$ and $\DD$.
  Then no reasonable (JEP)-expansion of $\VV^\fin$ has the Ramsey property.
\end{THM}
\begin{proof}
  Let $\VV$ be a nontrivial locally finite variety of lattices distinct from $\LL$ and $\DD$,
  and let $\WW$ be a reasonable (JEP)-expansion of $\VV^\fin$.
  Assume that $\WW$ has the Ramsey property. Because $\WW$ has (JEP) we know from \cite{Nesetril}
  that $\WW$ has~(AP).
  
  Let us show that $\VV^\fin$ has (AP).
  Take any $\calA, \calB_1, \calB_2 \in \VV^\fin$
  and embeddings $f_1 : \calA \hookrightarrow \calB_1$ and $f_2 : \calA \hookrightarrow \calB_2$. Because
  $\WW$ is a reasonable expansion of $\VV^\fin$, there exists a linear order $\sqsubset$
  such that $(\calA, \Boxed\sqsubset) \in \WW$, and then there exist linear orders
  $\sqsubset_1$ and $\sqsubset_2$ such that $(\calB_1, \Boxed{\sqsubset_1}), (\calB_2, \Boxed{\sqsubset_2}) \in \WW$
  and $f_1 : (\calA, \Boxed\sqsubset) \hookrightarrow (\calB_1, \Boxed{\sqsubset_1})$,
  $f_2 : (\calA, \Boxed\sqsubset) \hookrightarrow (\calB_2, \Boxed{\sqsubset_2})$ are embeddings.
  As we have just seen, $\WW$ has (AP), so there
  is a $(\calC, \Boxed\prec) \in \WW$ and embeddings $g_1 : (\calB_1, \Boxed{\sqsubset_1}) \hookrightarrow (\calC, \Boxed\prec)$
  and $g_2 : (\calB_2, \Boxed{\sqsubset_2}) \hookrightarrow (\calC, \Boxed\prec)$ such that $g_1 \circ f_1 = g_2 \circ f_2$.
  Then, clearly, $g_1 : \calB_1 \hookrightarrow \calC$ and $g_2 : \calB_2 \hookrightarrow \calC$
  are embeddings satisfying $g_1 \circ f_1 = g_2 \circ f_2$. This completes the proof that $\VV^\fin$ has (AP).
  
  It is a well known fact in lattice theory (see~\cite[Corollary~509]{gratzer-F}) that if $\XX$ is a locally finite
  variety of lattices then $\XX$ has (AP) if and only if $\XX^\fin$ has (AP).
  Therefore, $\VV$ has (AP). But by the famous result of Day and Je\v zek~\cite{day-jezek},
  the only nontrivial varieties of lattices with (AP) are $\LL$ and $\DD$. Contradiction.
\end{proof}

The main goal of this section is to show that every variety of lattices and every variety of semilattices
gives rise to a class of finite linearly ordered posets having both the Ramsey and the ordering property.
The idea is to replace the algebraic structure by the corresponding relational one. So, the structures we will be working with
are lattices understood as partially ordered sets expanded by linear orders that extend the partial order.

Let $\PP$ denote the class of all the posets (of all cardinalities).
A \emph{linearly ordered poset} is a structure $\calA = (A, \Boxed\sqsubseteq, \Boxed<)$ where
$(A, \Boxed\sqsubseteq)$ is a poset and $<$ is a linear order on $A$ which extends $\sqsubseteq$
(that is, if $a \sqsubseteq b$ and $a \ne b$ then $a < b$).
Let $\OPP$ denote the class of linearly ordered posets (of all cardinalities).
The Ramsey property for the class $\OPP^\fin$ was established in two steps:
first the ordering property for $\OPP^\fin$ was established in~\cite{PTW}, and then in~\cite{fouche}
the ordering property was used to prove the Ramsey property.

An alternative proof that the class $\OPP$ has the Ramsey property
was presented in~\cite{masul-rpopos}. Whereas the original proof in~\cite{PTW}
of the ordering property for $\OPP^\fin$ relies on the Dual Ramsey Theorem, the alternative
proof in~\cite{masul-rpopos} derives the the Ramsey property for $\OPP^\fin$
as a direct consequence of the Graham-Rothschild Theorem (Theorem~\ref{wtc.thm.G-R}).
Interestingly, the class $\OPP$ was the only known Ramsey class of structures where
the proof of the Ramsey property relied on proving first that the class has the ordering property.
The proof presented in~\cite{masul-rpopos} is new not only because new proof strategies were used,
but also because it does not not rely on the ordering property.

We shall now present an extract of this proof restructured so as to enable us
to reason about the ordering property not only of the class $\OPP$, but also some of its subclasses.
In this presentation we refrain from the explicit use of the machinery of category theory which was the main
language used in~\cite{masul-rpopos} and instead rephrase the proof in terms of first-order
structures and embeddings.

Let $\lex$, $\alex$ and $\clex$ denote the \emph{lexicographic}, \emph{anti-lexicographic} and
\emph{complemented lexicographic} ordering on $\calP(\{1, \ldots, n\})$, respectively, defined as follows:
\begin{align*}
A \lex B \text{ iff }  & A \subseteq B \text{ or }\\
                       & \min(B \setminus A) < \min(A \setminus B) \text{ in case $A$ and $B$ are incomparable},\\
A \alex B \text{ iff } & A \subseteq B \text{ or }\\
                       & \max(A \setminus B) < \max(B \setminus A) \text{ in case $A$ and $B$ are incomparable},\\
A \clex B \text{ iff } & \{1, \ldots, n\} \setminus A \lex \{1, \ldots, n\} \setminus B\\
          \text{ iff } & A \supseteq B \text{ or }\\
                       & \min(A \setminus B) < \min(B \setminus A) \text{ in case $A$ and $B$ are incomparable},
\end{align*}
where $<$ denotes the usual linear order on the integers. It is easy to see that all three are linear orders on
$\calP(\{1, \ldots, n\})$.

\begin{THM}\label{wtc.thm.1plus} (cf.\ \cite[Theorem~4.1]{masul-rpopos})
  For $n \in \NN$ let $\Pi_n$ denote the following linearly ordered poset:
  $
    \Pi_n = \big(\calP(\{1, \ldots, n\}), \Boxed\supseteq, \Boxed\clex\big)
  $.
  \begin{itemize}
  \item[$(a)$]
    For every $k \ge 2$ and all finite linearly ordered posets $\calA, \calB \in \OPP$ such that $\calA \hookrightarrow \calB$
    there is an $N \in \NN$ such that $\Pi_N \longrightarrow (\calB)^{\calA}_k$.
  \item[$(b)$]
    Let $\KK^*$ be a subclass of $\OPP$ such that $\Pi_n \in \KK^*$ for all $n \in \NN$. Then
    $\KK^*$ has the Ramsey property. In particular, $\OPP$ has the Ramsey property (see~\cite{PTW, fouche} for the original proof).
  \end{itemize}
\end{THM}
\begin{proof} (Sketch)
  $(a)$
  Let $k \ge 2$ and let $\calA = (A, \Boxed\sqsubseteq, \Boxed<)$ and $\calB = (B, \Boxed\sqsubseteq, \Boxed<)$
  be finite linearly ordered posets such that $\calA \hookrightarrow \calB$.
  
  A \emph{downset} in a poset $\calA$ is a subset $D \subseteq A$ such
  that $x \in D$ and $y \sqsubseteq x$ implies $y \in D$.
  For $a \in A$ let $\downarrow_\calA a = \{x \in A : x \sqsubseteq a\}$.
  Clearly, $\downarrow_\calA a$ is always a downset in $\calA$, but not all the downsets are of the form
  $\downarrow_\calA a$. To see this, take two $a, b \in A$ incomparable with respect to~$\sqsubseteq$.
  Then $\downarrow_\calA a \;\union \downarrow_\calA b$
  is a downset in $\calA$ which is not of the form $\downarrow_\calA x$ for some $x \in A$.
  
  Let $\calA$ have $m_\calA$ nonempty downsets and let $\calB$ have $m_\calB$ nonempty downsets.
  According to the Graham-Rothschild Theorem (Theorem~\ref{wtc.thm.G-R}) there exists an $N$ such that for every partition
  $W^N_{m_\calA}(\{0\}) = \calX_1 \union \ldots \union \calX_k$ there is a $u \in W^N_{m_\calB}(\{0\})$ and a $j$ satisfying
  $\{u \cdot h : h \in W^{m_\calB}_{m_\calA}(\{0\})\} \subseteq \calX_j$.
  Let us show that $\Pi_N \longrightarrow (\calB)^{\calA}_k$.
  
  Let $D_1$, \ldots, $D_{m_\calA}$ be all the nonempty downsets in $\calA$ and let
  $D_1 \alex D_2 \alex \ldots \alex D_m$.
  For $u \in W^n_m(\{0\})$, let $X_i = u^{-1}(x_i)$, $1 \le i \le m$, and let
  $a_i = \UNION\{X_\alpha : i \in D_\alpha\}$, $1 \le i \le k$.
  It was shown in the proof of \cite[Theorem~4.1]{masul-rpopos} that
  for every $u \in W^n_m(\{0\})$ the mapping $\Phi_{\calA, n}(u) : \calA \to \Pi_n : i \mapsto a_i$ is an embedding.
  
  Take any partition $\binom{\Pi_N}{\calA} = \calX_1 \union \ldots \union \calX_k$.
  Let $W^N_{m_\calA}(\{0\}) = \calX'_1 \union \ldots \union \calX'_k$
  be a partition constructed as follows:
  for $w \in W^N_{m_\calA}(\{0\})$ let
  \begin{equation}\label{04.eq.2}
    w \in \calX'_j \text{ if and only if } \im(\Phi_{\calA, N}(w)) \in \calX_j
  \end{equation}
  (here, $\im(f)$ denotes the image of $f$ as a substructure of its codomain).
  By the construction of $N$, there exist a $u \in W^N_{m_\calB}(\{0\})$ and a $j$ such that
  \begin{equation}\label{04.eq.1}
    \{u \cdot h : h \in W^{m_\calB}_{m_\calA}(\{0\})\} \subseteq \calX'_j.
  \end{equation}
  Then $\tilde\calB = \im(\Phi_{\calB, N}(u))$ is a copy of $\calB$ in $\Pi_N$ because $\Phi_{\calB, N}(u)$ is an embedding.
  Let us show that $\binom{\tilde\calB}{\calA} \subseteq \calX_j$.

  Take any $\tilde\calA \in \binom{\tilde\calB}{\calA}$. Then there is an embedding $f : \calA \hookrightarrow \calB$ such that
  $\im(\Phi_{\calB, N}(u) \circ f) = \tilde\calA$. It was shown in the proof of \cite[Theorem~4.1]{masul-rpopos} that
  one can then find
  a word $h = h_1 h_2 \ldots h_{m_\calB} \in W^{m_\calB}_{m_\calA}(\{0\})$ such that
  $\Phi_{\calB, N}(u) \circ f = \Phi_{\calA, N}(u \cdot h)$.
  From \eqref{04.eq.1} we know that $u \cdot h \in \calX'_j$, whence
  $\im(\Phi_{\calA, N}(u \cdot h)) \in \calX_j$ by~\eqref{04.eq.2}.
  Therefore, $\im(\Phi_{\calB, N}(u) \circ f) = \tilde\calA \in \calX_j$.
  
  $(b)$ Directly from $(a)$.
\end{proof}

Using the standard Sierpinski-style coloring obtained by comparing two linear orders
we can deduce the ordering property for a class from the fact that it has the Ramsey property.

\begin{THM}\label{wtc.thm.op}
  Let $\KK^*$ be a subclass of $\OPP$ such that $\Pi_n \in \KK^*$ for all $n \in \NN$. Then $\KK^*$ has the ordering property.
  In particular, the class $\OPP$ has the ordering property (see~\cite{PTW,sokic} for the original proof).
\end{THM}
\begin{proof}
  Let $\KK^*$ be a subclass of $\OPP$ such that $\Pi_n \in \KK^*$ for all $n \in \NN$.
  It is easy to see that (OP) is equivalent to the following whenever $\KK^*$ has (JEP):
  \begin{description}
  \item[(OP')]
    for every $(\calA, \Boxed{<}) \in \KK^*$ there is a $\calB \in \KK$ such that $(\calA, \Boxed{<}) \embedsto (\calB, \Boxed\sqsubset)$
    for every linear order $\sqsubset$ on $B$ with $(\calB, \Boxed\sqsubset) \in \KK^*$.
    We say that $\calB$ is a \emph{witness for the ordering property for $(\calA, \Boxed{<})$}.
  \end{description}
  Let us show that $\KK^*$ has (JEP) so that we can use (OP').
  One of the byproducts of Theorem~\ref{wtc.thm.1plus}~$(a)$ is the following: for every $\calA \in \OPP$ there is an
  $n \in \NN$ such that $\calA \hookrightarrow \Pi_n$. So, take any $\calA, \calB \in \KK^*$. As we have just seen,
  $\calA \hookrightarrow \Pi_n$ and $\calB \hookrightarrow \Pi_m$ for some $n, m \in \NN$. Without loss of generality
  we can take that $n \le m$. Therefore, $\calA \hookrightarrow \Pi_n \hookrightarrow \Pi_m \hookleftarrow \calB$.

  So, let us show (OP') for $\KK^*$.
  Let $\calB = (B, \Boxed\sqsubseteq, \Boxed<)$ be a finite linearly ordered poset. If $(B, \Boxed\sqsubseteq)$ is an antichain
  then $(B, \Boxed\sqsubseteq)$ is a witness for the ordering property for $\calB$. Assume, therefore, that
  $\calB$ is not an antichain and take any $x, y \in B$ such that $x \sqsubset y$. Add a new element $z \notin B$
  to $\calB$ to obtain a finite linearly ordered poset $\calB_1 = (B_1, \Boxed{\sqsubseteq_1}, \Boxed{<_1})$ as follows:
  $B_1 = B \union \{z\}$, $\Boxed{\sqsubseteq_1} = \Boxed\sqsubseteq \union \{(z,z)\}$ 
  (in other words, $z$ is incomparable with every $b \in B$), and $<_1$ is an extension of $<$ such that
  $x \mathrel{<_1} z \mathrel{<_1} y$.
  Let $\calA = (\{0,1\}, \Boxed=, \Boxed<)$ be a two-element linearly ordered antichain ($0 < 1$).
  By Theorem~\ref{wtc.thm.1plus}~$(a)$ there is an $N \in \NN$ such that $\Pi_N \longrightarrow (\calB)^\calA_2$.
  For notational convenience, let $\Pi_N = (\pi_N, \Boxed{\sqsubseteq_N}, \Boxed{<_N})$.
  Let us show that $(\pi_N, \Boxed{\sqsubseteq_N}) \in \KK$ is a witness for the ordering property for $\calB_1$
  and hence a witness for the ordering property for $\calB$ since $\calB \hookrightarrow \calB_1$.
  Take any linear order $\prec$ on $\pi_N$ which extends $\sqsubseteq_N$ and consider the coloring
  $\binom{\Pi_N}{\calA} = X_0 \union X_1$ as follows.
  Let $\tilde\calA = (\{\tilde 0, \tilde 1\}, \Boxed=, \Boxed<) \in \binom{\Pi_N}{\calA}$
  where $\tilde 0 < \tilde 1$. Put
  \begin{align*}
    \tilde A \in X_0 \text{\quad if\quad} &\tilde 0 \mathrel{<_N} \tilde 1 \text{ and } \tilde 0 \prec \tilde 1, \text{ or }
                                    \tilde 1 \mathrel{<_N} \tilde 0 \text{ and } \tilde 1 \prec \tilde 0;\\
    \tilde A \in X_1 \text{\quad  if\quad} &\tilde 0 \mathrel{<_N} \tilde 1 \text{ and } \tilde 1 \prec \tilde 0, \text{ or }
                                    \tilde 1 \mathrel{<_N} \tilde 0 \text{ and } \tilde 0 \prec \tilde 1.
  \end{align*}
  Then there is a $\tilde \calB_1 \in \binom{\calC}{\calB_1}$ such that $\binom{\tilde\calB_1}{\calA}$ is monochromatic.
  Let us show that $\binom{\tilde\calB_1}{\calA} \subseteq X_0$. Suppose, to the contrary, that
  $\binom{\tilde\calB_1}{\calA} \subseteq X_1$. Let $\tilde x, \tilde y, \tilde z$ be the elements of $\tilde\calB_1$
  which correspond to $x, y, z$ in $\calB_1$. Then $\{\tilde x, \tilde z\}$ and $\{\tilde y, \tilde z\}$ are antichains in $\calC$
  such that $\tilde x \mathrel{<_N} \tilde z \mathrel{<_N} \tilde y$ (because this is their order in $\calB_1$) and
  $\tilde x \succ \tilde z \succ \tilde y$ (because $\binom{\tilde\calB_1}{\calA} \subseteq X_1$).
  On the other hand, $\tilde x \mathrel{\sqsubset_N} \tilde y$ -- contradiction with the fact that $\prec$
  extends $\sqsubseteq_N$.
  
  Therefore, $\binom{\tilde\calB_1}{\calA} \subseteq X_0$, which means that $<_N$ and $\prec$ coincide on all 2-element
  antichains of $\tilde\calB_1$. We already know that $<_N$ and $\prec$ coincide on all comparable pairs of elements because $\prec$
  extends $\sqsubseteq_N$. So, $\tilde\calB_1 \le (\pi_N, \Boxed{\sqsubseteq_N}, \Boxed{\prec})$, whence
  $\calB_1 \hookrightarrow (\pi_N, \Boxed{\sqsubseteq_N}, \Boxed{\prec})$.
\end{proof}

With all the technical results in place, we conclude the section by showing that the
``relational alter ego'' of every variety of lattices and every variety of semilattices has both the Ramsey property and
the ordering property. Let us start by introducing a bit of notation.
For a finite lattice $\calL = (L, \Boxed\land, \Boxed\lor)$ let $\rel(\calL) = (L, \Boxed\sqsubseteq)$ be the
corresponding finite poset (where we set $a \sqsubseteq b$ if and only if $a \land b = a$).
With the slight abuse of set notation, for a variety $\VV$ of lattices
let $\rel(\VV^\fin) = \{\rel(\calL) : \calL \in \VV^\fin\}$. On the other hand, let
$\overrightarrow\rel(\VV^\fin)$ denote the class of all the finite linearly ordered posets $(L, \Boxed\sqsubseteq, \Boxed\prec)$
where $(L, \Boxed\sqsubseteq) = \rel(\calL)$ for some $\calL \in \VV^\fin$.
So, the elements of $\overrightarrow\rel(\VV^\fin)$ are all the posets from $\rel(\VV^\fin)$ where each poset is
expanded by all possible linear extensions of the partial order.

Analogously, for a finite semilattice $\calS = (S, \Boxed\land)$ let $\rel(\calS) = (S, \Boxed\sqsubseteq)$ be the
corresponding finite poset(where we set $a \sqsubseteq b$ if and only if $a \land b = a$),
let $\rel(\VV^\fin) = \{\rel(\calS) : \calS \in \VV^\fin\}$ for a variety $\VV$, and let
$\overrightarrow\rel(\VV^\fin)$ denote the class of all the finite linearly ordered posets $(S, \Boxed\sqsubseteq, \Boxed\prec)$
where $(S, \Boxed\sqsubseteq) = \rel(\calS)$ for some $\calS \in \VV^\fin$.

\begin{THM}
  $(a)$ Let $\VV$ be a nontrivial variety of lattices.
  Then $\overrightarrow\rel(\VV^\fin)$ has both the Ramsey property and the ordering property.

  $(b)$ Let $\VV$ be a nontrivial variety of semilattices.
  Then $\overrightarrow\rel(\VV^\fin)$ has both the Ramsey property and the ordering property.
\end{THM}
\begin{proof}
  $(a)$
  Every nontrivial variety of lattices contains the two-element lattice $\calL_2 = (\{0, 1\}, \Boxed\land, \Boxed\lor)$
  where $0 < 1$, and hence all the finite powers of $\calL_2$.
  Since $(\calP(A), \Boxed\cup, \Boxed\cap) \cong \calL_2^{|A|}$ for every
  finite set $A$, it follows that every nontrivial variety of lattices contains all the lattices of the form
  $(\calP(A), \Boxed\cup, \Boxed\cap)$ where $A$ is a finite set. Therefore,
  $\Pi_n \in \overrightarrow\rel(\VV^\fin)$ for all $n \in \NN$.
  It is now immediate from Theorems~\ref{wtc.thm.1plus} and~\ref{wtc.thm.op} that
  $\overrightarrow\rel(\VV^\fin)$ has the Ramsey property as well as the ordering property.

  $(b)$
  Analogous to $(a)$.
\end{proof}

\section{The weak triangle condition}
\label{wtc.sec.wtc}

The main idea used to prove Theorem~\ref{wtc.thm.op} (adding a strategically placed triangle)
can be generalized to a much larger class of structures, those which satisfy what we call the
\emph{weak triangle condition}. This is a weaker form of the \emph{triangle condition}
introduced in~\cite{KPT} also in connection to understanding the ordering property.

Let $\Theta$ be a first-order language and let $\Boxed< \notin \Theta$
be a new binary relational symbol. Let $\Theta^* = \Theta \union \{\Boxed<\}$.
As usual, given a $\Theta^*$-structure $\calA$ we shall always interpret $<$ in $\calA$ as a linear order on $A$.
Let $\KK^*$ be a class of finite $\Theta^*$-structures and assume that it
is an order expansion of a class $\KK$ of $\Theta$-structures.
Let $S_2(\KK^*)$ denote the class of all the 2-generated structures in $\KK^*$.
We say that $\KK^*$ has the \emph{weak triangle condition} if
\begin{description}
\item[$(W\triangle C)$]
  for every nonempty finite $\Sigma \subseteq S_2(\KK^*)$ there is a $\tau \in S_2(\KK^*)$ such that
  for every $\sigma \in \Sigma$ there exists a $(\calD, \Boxed{<_D}) \in \KK^*$ and $x, y, z \in D$
  such that $x \mathrel{<_D} y \mathrel{<_D} z$, $\langle x, z \rangle_{(\calD, \Boxed{<_D})} \cong \sigma$ and
  $\langle x, y \rangle_{(\calD, \Boxed{<_D})} \cong \langle y, z \rangle_{(\calD, \Boxed{<_D})} \cong \tau$.
\end{description}

\begin{THM}\label{wtc.thm.RP-OP}
  Let $\Theta$ be a first-order language and let $\Boxed< \notin \Theta$
  be a new binary relational symbol. Let $\Theta^* = \Theta \union \{\Boxed<\}$.
  Let $\KK^*$ be a Ramsey age of finite $\Theta^*$-structures which
  has the weak triangle condition. Then $\KK^*$ has the ordering property.
\end{THM}
\begin{proof}
  Before we move on to the proof, let us describe a construction that this proof relies on.
  Take any $(\calB, \Boxed{<_B}) \in \KK^*$ and let $(a_i, b_i)$, $1 \le i \le n$, be an
  enumeration of all the pairs of elements of $B$ satisfying $a_i \mathrel{<_B} b_i$. Let $\sigma_i =
  \langle a_i, b_i \rangle_{(\calB, \Boxed{<_B})}$, $1 \le i \le n$.
  Put $\Sigma = \{\sigma_i : 1 \le i \le n\}$.
  By $(W\triangle C)$ there exists a $\tau \in S_2(\KK^*)$,
  and for every $i \in \{1, \ldots, n\}$ there exist $(\calD_i, \Boxed{<_i}) \in \KK^*$ and $x_i, y_i, z_i \in D_i$
  such that $x_i \mathrel{<_i} y_i \mathrel{<_i} z_i$, $\langle x_i, z_i \rangle_{(\calD_i, \Boxed{<_i})} \cong \sigma_i$ and
  $\langle x_i, y_i \rangle{(\calD_i, \Boxed{<_i})} \cong \langle y_i, z_i \rangle_{(\calD_i, \Boxed{<_i})} \cong \tau$.
  We now perform $n$ amalgamations inductively
  as follows. Put $(\calB_0, \Boxed{<_{B_0}}) = (\calB, \Boxed{<_B})$.
  Assume that $(\calB_{i-1}, \Boxed{<_{B_{i-1}}})$ has been constructed. In the $i$th step
  amalgamate $(\calB_{i-1}, \Boxed{<_{B_{i-1}}})$ with $(\calD_i, \Boxed{<_i})$ over $\sigma_i$ and embeddings
  $f_i : \sigma_i \hookrightarrow (\calB_{i-1}, \Boxed{<_{B_{i-1}}}) : a_i \mapsto a_i, b_i \mapsto b_i$ and
  $g_i : \sigma_i \hookrightarrow (\calD_i, \Boxed{<_i}) : a_i \mapsto x_i, b_i \mapsto z_i$, and denote the amalgam by
  $(\calB_i, \Boxed{<_{B_i}})$, Fig.~\ref{wtc.fig.nabla}.
  Without loss of generality we can assume that $(\calB, \Boxed{<_B}) \le (\calB_i, \Boxed{<_{B_i}})$ so that
  the procedure can continue as described. Let us denote the final amalgam $(\calB_n, \Boxed{<_{B_n}})$ by $\nabla (\calB, \Boxed{<_B})$.
  Note that $\nabla (\calB, \Boxed{<_B}) \in \KK^*$.
  Without loss of generality we can assume that $(\calB, \Boxed{<_B}) \le \nabla(\calB, \Boxed{<_B})$.
  
  \begin{figure}
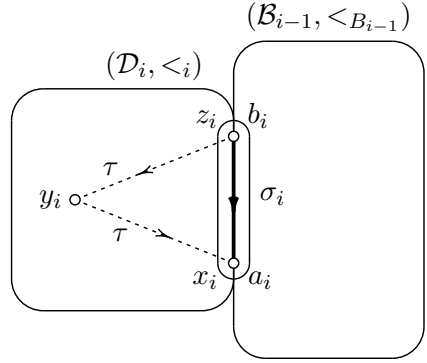

    \centering
    %% DO NOT FORGET TO \usepackage{pgf,acadpgf}
\begin{pgfpicture}
  \pgfsetxvec{\pgfpoint{\acadpgfunit}{0pt}}
  \pgfsetyvec{\pgfpoint{0pt}{\acadpgfunit}}
  \pgfsetlinewidth{\acadpgflinewidth}
  \pgftransformshift{\pgfpointxy{62.5}{100.0}}

  \begin{pgfscope}
    \pgfsetdash{{1.5pt}{2pt}}{0pt}
    \pgfpathmoveto{\pgfpointxy{500.0}{500.0}}
    \pgfpathlineto{\pgfpointxy{250.0}{400.0}}
    \pgfusepath{stroke}
  \end{pgfscope}
  \begin{pgfscope}
    \pgfsetdash{{1.5pt}{2pt}}{0pt}
    \pgfpathmoveto{\pgfpointxy{250.0}{400.0}}
    \pgfpathlineto{\pgfpointxy{500.0}{300.0}}
    \pgfusepath{stroke}
  \end{pgfscope}
  \begin{pgfscope}
    \pgfpathmoveto{\pgfpointxy{381.348}{353.954}}
    \pgfpatharcaxes{218.199}{248.199}{\pgfpointxy{45.0}{0.0}}{\pgfpointxy{0.0}{45.0}}
    \pgfusepath{stroke}
  \end{pgfscope}
  \begin{pgfscope}
    \pgfpathmoveto{\pgfpointxy{400.0}{340.0}}
    \pgfpatharcaxes{68.1986}{98.1986}{\pgfpointxy{45.0}{0.0}}{\pgfpointxy{0.0}{45.0}}
    \pgfusepath{stroke}
  \end{pgfscope}
  \begin{pgfscope}
    \pgfpathmoveto{\pgfpointxy{373.13}{442.759}}
    \pgfpatharcaxes{81.8014}{111.801}{\pgfpointxy{45.0}{0.0}}{\pgfpointxy{0.0}{45.0}}
    \pgfusepath{stroke}
  \end{pgfscope}
  \begin{pgfscope}
    \pgfpathmoveto{\pgfpointxy{350.0}{440.0}}
    \pgfpatharcaxes{291.801}{321.801}{\pgfpointxy{45.0}{0.0}}{\pgfpointxy{0.0}{45.0}}
    \pgfusepath{stroke}
  \end{pgfscope}
  \begin{pgfscope}
    \pgfpathmoveto{\pgfpointxy{500.0}{523.062}}
    \pgfpathlineto{\pgfpointxy{500.0}{276.938}}
    \pgfusepath{stroke}
  \end{pgfscope}
  \begin{pgfscope}
    \pgfpathmoveto{\pgfpointxy{500.0}{600.0}}
    \pgfpathlineto{\pgfpointxy{500.0}{200.0}}
    \pgfusepath{stroke}
  \end{pgfscope}
  \begin{pgfscope}
    \pgfpathmoveto{\pgfpointxy{500.0}{200.0}}
    \pgfpatharcaxes{180.0}{270.0}{\pgfpointxy{50.0}{0.0}}{\pgfpointxy{0.0}{50.0}}
    \pgfusepath{stroke}
  \end{pgfscope}
  \begin{pgfscope}
    \pgfpathmoveto{\pgfpointxy{550.0}{150.0}}
    \pgfpathlineto{\pgfpointxy{750.0}{150.0}}
    \pgfusepath{stroke}
  \end{pgfscope}
  \begin{pgfscope}
    \pgfpathmoveto{\pgfpointxy{750.0}{150.0}}
    \pgfpatharcaxes{-90.0}{0.0}{\pgfpointxy{50.0}{0.0}}{\pgfpointxy{0.0}{50.0}}
    \pgfusepath{stroke}
  \end{pgfscope}
  \begin{pgfscope}
    \pgfpathmoveto{\pgfpointxy{800.0}{200.0}}
    \pgfpathlineto{\pgfpointxy{800.0}{600.0}}
    \pgfusepath{stroke}
  \end{pgfscope}
  \begin{pgfscope}
    \pgfpathmoveto{\pgfpointxy{800.0}{600.0}}
    \pgfpatharcaxes{0.0}{90.0}{\pgfpointxy{50.0}{0.0}}{\pgfpointxy{0.0}{50.0}}
    \pgfusepath{stroke}
  \end{pgfscope}
  \begin{pgfscope}
    \pgfpathmoveto{\pgfpointxy{750.0}{650.0}}
    \pgfpathlineto{\pgfpointxy{550.0}{650.0}}
    \pgfusepath{stroke}
  \end{pgfscope}
  \begin{pgfscope}
    \pgfpathmoveto{\pgfpointxy{550.0}{650.0}}
    \pgfpatharcaxes{90.0}{180.0}{\pgfpointxy{50.0}{0.0}}{\pgfpointxy{0.0}{50.0}}
    \pgfusepath{stroke}
  \end{pgfscope}
  \begin{pgfscope}
    \pgfsetlinewidth{0.50mm}
    \pgfpathmoveto{\pgfpointxy{500.0}{492.0}}
    \pgfpathlineto{\pgfpointxy{500.0}{308.0}}
    \pgfusepath{stroke}
  \end{pgfscope}
  \begin{pgfscope}
    \pgfsetlinewidth{0.50mm}
    \pgfpathmoveto{\pgfpointxy{506.029}{404.1}}
    \pgfpatharcaxes{150.0}{180.0}{\pgfpointxy{45.0}{0.0}}{\pgfpointxy{0.0}{45.0}}
    \pgfusepath{stroke}
  \end{pgfscope}
  \begin{pgfscope}
    \pgfsetlinewidth{0.50mm}
    \pgfpathmoveto{\pgfpointxy{500.0}{381.6}}
    \pgfpatharcaxes{0.0}{30.0}{\pgfpointxy{45.0}{0.0}}{\pgfpointxy{0.0}{45.0}}
    \pgfusepath{stroke}
  \end{pgfscope}
  \begin{pgfscope}
    \pgfpathmoveto{\pgfpointxy{525.0}{500.0}}
    \pgfpathlineto{\pgfpointxy{525.0}{300.0}}
    \pgfusepath{stroke}
  \end{pgfscope}
  \begin{pgfscope}
    \pgfpathmoveto{\pgfpointxy{475.0}{300.0}}
    \pgfpathlineto{\pgfpointxy{475.0}{500.0}}
    \pgfusepath{stroke}
  \end{pgfscope}
  \begin{pgfscope}
    \pgfpathmoveto{\pgfpointxy{450.0}{575.0}}
    \pgfpathlineto{\pgfpointxy{200.0}{575.0}}
    \pgfusepath{stroke}
  \end{pgfscope}
  \begin{pgfscope}
    \pgfpathmoveto{\pgfpointxy{450.0}{225.0}}
    \pgfpathlineto{\pgfpointxy{200.0}{225.0}}
    \pgfusepath{stroke}
  \end{pgfscope}
  \begin{pgfscope}
    \pgfpathmoveto{\pgfpointxy{150.0}{525.0}}
    \pgfpathlineto{\pgfpointxy{150.0}{275.0}}
    \pgfusepath{stroke}
  \end{pgfscope}
  \begin{pgfscope}
    \pgfpathmoveto{\pgfpointxy{200.0}{575.0}}
    \pgfpatharcaxes{90.0}{180.0}{\pgfpointxy{50.0}{0.0}}{\pgfpointxy{0.0}{50.0}}
    \pgfusepath{stroke}
  \end{pgfscope}
  \begin{pgfscope}
    \pgfpathmoveto{\pgfpointxy{150.0}{275.0}}
    \pgfpatharcaxes{180.0}{270.0}{\pgfpointxy{50.0}{0.0}}{\pgfpointxy{0.0}{50.0}}
    \pgfusepath{stroke}
  \end{pgfscope}
  \begin{pgfscope}
    \pgfpathmoveto{\pgfpointxy{450.0}{225.0}}
    \pgfpatharcaxes{-90.0}{0.0}{\pgfpointxy{50.0}{0.0}}{\pgfpointxy{0.0}{50.0}}
    \pgfusepath{stroke}
  \end{pgfscope}
  \begin{pgfscope}
    \pgfpathmoveto{\pgfpointxy{500.0}{525.0}}
    \pgfpatharcaxes{0.0}{90.0}{\pgfpointxy{50.0}{0.0}}{\pgfpointxy{0.0}{50.0}}
    \pgfusepath{stroke}
  \end{pgfscope}
  \begin{pgfscope}
    \pgfpathmoveto{\pgfpointxy{525.0}{500.0}}
    \pgfpatharcaxes{0.0}{180.0}{\pgfpointxy{25.0}{0.0}}{\pgfpointxy{0.0}{25.0}}
    \pgfusepath{stroke}
  \end{pgfscope}
  \begin{pgfscope}
    \pgfpathmoveto{\pgfpointxy{475.0}{300.0}}
    \pgfpatharcaxes{-180.0}{0.0}{\pgfpointxy{25.0}{0.0}}{\pgfpointxy{0.0}{25.0}}
    \pgfusepath{stroke}
  \end{pgfscope}
  \begin{pgfscope}
    \pgfsetfillcolor{white}
    \pgfpathellipse{\pgfpointxy{500.0}{500.0}}{\pgfpointxy{8.0}{0.0}}{\pgfpointxy{0.0}{8.0}}
    \pgfusepath{fill,stroke}
  \end{pgfscope}
  \begin{pgfscope}
    \pgfsetfillcolor{white}
    \pgfpathellipse{\pgfpointxy{500.0}{300.0}}{\pgfpointxy{8.0}{0.0}}{\pgfpointxy{0.0}{8.0}}
    \pgfusepath{fill,stroke}
  \end{pgfscope}
  \begin{pgfscope}
    \pgfsetfillcolor{white}
    \pgfpathellipse{\pgfpointxy{250.0}{400.0}}{\pgfpointxy{8.0}{0.0}}{\pgfpointxy{0.0}{8.0}}
    \pgfusepath{fill,stroke}
  \end{pgfscope}
  \pgftext[bottom,left,at={\pgfpointxy{521.795}{513.924}}]{$b_i$}
  \pgftext[bottom,right,at={\pgfpointxy{475.27}{513.804}}]{$z_i$}
  \pgftext[top,left,at={\pgfpointxy{523.109}{286.695}}]{$a_i$}
  \pgftext[top,right,at={\pgfpointxy{475.466}{286.015}}]{$x_i$}
  \pgftext[right,at={\pgfpointxy{230.0}{400.0}}]{$y_i$}
  \pgftext[left,at={\pgfpointxy{541.273}{405.695}}]{$\sigma_i$}
  \pgftext[bottom,at={\pgfpointxy{650.0}{662.0}}]{$(\calB_{i-1}, \Boxed{<_{B_{i-1}}})$}
  \pgftext[bottom,right,at={\pgfpointxy{452.048}{588.888}}]{$(\calD_i, \Boxed{<_i})$}
  \pgftext[top,right,at={\pgfpointxy{334.073}{355.171}}]{$\tau$}
  \pgftext[bottom,right,at={\pgfpointxy{318.985}{438.794}}]{$\tau$}
\end{pgfpicture}
    \caption{The $\nabla$ construction}
    \label{wtc.fig.nabla}
  \end{figure}
  
  Take any $(\calA, \Boxed{<_A}) \in \KK^*$.

  \medskip
  
  Case 1: $(\calA, \Boxed{>_A}) \notin \KK^*$.
  
  \medskip

  Let $(\calC, \Boxed{<_C}) = \nabla (\calA, \Boxed{<_A})$. By the assumption, $(\calA, \Boxed{<_A}) \le (\calC, \Boxed{<_C})$.
  Since $K^*$ has the Ramsey property there is a $(\calQ, \Boxed{<_Q}) \in \KK^*$ such that
  $(\calQ, \Boxed{<_Q}) \longrightarrow (\calC, \Boxed{<_C})^\tau_2$.
  Let us show that $\calQ$ is the witness for the
  ordering property for $(\calA, \Boxed{<_A})$. Take any linear order $\sqsubset$ on $Q$ such that $(\calQ, \Boxed\sqsubset) \in \KK^*$.
  Consider the following coloring $\binom{(\calQ, \Boxed{<_Q})}{\tau} = \calX_1 \union \calX_2$: for $q, r \in Q$ such that
  $q \mathrel{<_Q} r$ and $\langle q, r \rangle_{(\calQ, \Boxed{<_Q})} \cong \tau$ put
  $\langle q, r \rangle_{(\calQ, \Boxed{<_Q})} \in \calX_1$ if $q \sqsubset r$, and put
  $\langle q, r \rangle_{(\calQ, \Boxed{<_Q})} \in \calX_2$ if $q \sqsupset r$.
  Then there is a monochromatic copy $(\tilde \calC, \Boxed{<_{\tilde C}})$ of $(\calC, \Boxed{<_C})$ in $(\calQ, \Boxed{<_Q})$.
  Let $(\tilde \calA, \Boxed{<_{\tilde A}})$ be a copy of $(\calA, \Boxed{<_A})$ in $(\tilde \calC, \Boxed{<_{\tilde C}})$.
  
  Let us first show that $\binom{(\tilde \calC, \Boxed{<_{\tilde C}})}{\tau} \subseteq \calX_2$ cannot happen.
  Assume, to the contrary, that this is the case. Then we can show that
  for all $\tilde a, \tilde b \in \tilde A$
  we have that $\tilde a \mathrel{<_{\tilde A}} \tilde b$ if and only if $\tilde a \sqsupset \tilde b$.
  Assume that $\tilde a \mathrel{<_{\tilde A}} \tilde b$. Then $(\tilde a, \tilde b)$ is a copy in $(\tilde \calA, \Boxed{<_{\tilde A}})$
  of some pair $(a_i, b_i)$ (see the beginning of the $\nabla$ construction). By the construction, there is a $\tilde y \in \tilde C$
  such that $\tilde a \mathrel{<_{\tilde C}} \tilde y \mathrel{<_{\tilde C}} \tilde b$ and
  $\langle \tilde a, \tilde y \rangle_{(\tilde \calC, \Boxed{<_{\tilde C}})} \cong
  \langle \tilde y, \tilde b \rangle_{(\tilde \calC, \Boxed{<_{\tilde C}})} \cong \tau$, see Fig.~\ref{wtc.fig.nabla}.
  From $\binom{(\tilde \calC, \Boxed{<_{\tilde C}})}{\tau} \subseteq \calX_2$ it now follows
  that $\tilde a \sqsupset \tilde y \sqsupset \tilde b$. But, then we have that
  $(\tilde \calA, \Boxed{>_{\tilde A}}) \cong \langle \tilde A \rangle_{(\calQ, \Boxed\sqsubset)}$ whence
  $(\calA, \Boxed{>_A}) \hookrightarrow (\calQ, \Boxed{\sqsubset})$. Since $(\calQ, \Boxed{\sqsubset}) \in \KK^*$
  we conclude $(\calA, \Boxed{>_A}) \in \KK^*$. Contradiction.

  Therefore, $\binom{(\tilde \calC, \Boxed{<_{\tilde C}})}{\tau} \subseteq \calX_1$.
  We can now repeat the same argument to show that
  $(\tilde \calA, \Boxed{<_{\tilde A}}) \cong \langle \tilde A \rangle_{(\calQ, \Boxed\sqsubset)}$ whence
  $(\calA, \Boxed{<_A}) \hookrightarrow (\calQ, \Boxed{\sqsubset})$.

  \medskip
  
  Case 2: $(\calA, \Boxed{>_A}) \in \KK^*$.
  
  \medskip

  Let $(\calB, \Boxed{<_B})$ be a structure which embeds both $(\calA, \Boxed{<_A})$ and $(\calA, \Boxed{>_A})$.
  Then $(\calA, \Boxed{<_A}) \hookrightarrow (\calB, \Boxed{<_{B}})$ and
  $(\calA, \Boxed{<_A}) \hookrightarrow (\calB, \Boxed{>_{B}})$.
  Let $(\calC, \Boxed{<_C}) = \nabla (\calB, \Boxed{<_B})$.
  By the Ramsey property there is a $(\calQ, \Boxed{<_Q}) \in \KK^*$ such that
  $(\calQ, \Boxed{<_Q}) \longrightarrow (\calC, \Boxed{<_C})^\tau_2$. Let us show that $\calQ$ is the witness for the
  ordering property for $(\calA, \Boxed{<_A})$. Take any linear order $\sqsubset$ on $Q$ such that $(\calQ, \Boxed\sqsubset) \in \KK^*$
  and construct the coloring $\binom{(\calQ, \Boxed{<_Q})}{\tau} = \calX_1 \union \calX_2$ as in Case~1.
  Then there is a monochromatic copy $(\tilde \calC, \Boxed{<_{\tilde C}})$ of $(\calC, \Boxed{<_C})$ in $(\calQ, \Boxed{<_Q})$.
  Let $(\tilde \calB, \Boxed{<_{\tilde B}})$ be a copy of $(\calB, \Boxed{<_B})$ in $(\tilde \calC, \Boxed{<_{\tilde C}})$.

  If $\binom{(\tilde \calC, \Boxed{<_{\tilde C}})}{\tau} \subseteq \calX_1$
  we can repeat the argument from Case 1 to show that
  $(\tilde \calB, \Boxed{<_{\tilde B}}) \cong \langle \tilde B \rangle_{(\calQ, \Boxed\sqsubset)}$ whence
  $(\calA, \Boxed{<_A}) \hookrightarrow (\tilde \calB, \Boxed{<_{\tilde B}}) \hookrightarrow (\calQ, \Boxed{\sqsubset})$.
  If, however, $\binom{(\tilde \calC, \Boxed{<_{\tilde C}})}{\tau} \subseteq \calX_2$ then
  $(\tilde \calB, \Boxed{>_{\tilde B}}) \cong \langle \tilde B \rangle_{(\calQ, \Boxed\sqsubset)}$ whence
  $(\calA, \Boxed{<_A}) \hookrightarrow (\tilde \calB, \Boxed{>_{\tilde B}}) \hookrightarrow (\calQ, \Boxed{\sqsubset})$.
\end{proof}

As an application of the above theorem let us consider structures with several poset relations (cf.~\cite{draganic-masul-mulitposets}).
Let $\calT = (\{1, \ldots, n\}, \Boxed\preccurlyeq)$ be a poset, $n \ge 1$, which we refer to as a \emph{template}.
A structure $(A, \Boxed\sqsubseteq_1, \ldots, \Boxed\sqsubseteq_n)$ consisting of $n$ partial orders
\emph{conforms to the template $\calT$} if $(\sqsubseteq_i) \subseteq (\sqsubseteq_j)$ whenever $i \preccurlyeq j$ in~$\calT$;
and is \emph{consistent} if there is a linear order on $A$ which extends each $\sqsubseteq_i$, $1 \le i \le n$.
(Note that $(A, \Boxed\sqsubseteq_1, \ldots, \Boxed\sqsubseteq_n)$ is consistent if and only if there do not exist
distinct $a, b$ and distinct $i, j$ such that $a \mathrel{\sqsubset_i} b$ and $b \mathrel{\sqsubset_j} a$.)
Given a template $\calT$, let $\PP_\calT$ denote the class of all the structures $(A, \Boxed\sqsubseteq_1, \ldots, \Boxed\sqsubseteq_n)$
(of all cardinalities) consisting of $n$ partial orders which conform to $\calT$ and which are consistent; and let
$\OPP_\calT$ denote the class of all the structures $(A, \Boxed\sqsubseteq_1, \ldots, \Boxed\sqsubseteq_n, \Boxed<)$
such that $(A, \Boxed\sqsubseteq_1, \ldots, \Boxed\sqsubseteq_n) \in \PP_\calT$ and $<$ is a linear order
which extends each $\sqsubseteq_i$, $1 \le i \le n$. (Note that $\PP(1) = \PP$ and $\OPP(1) = \OPP$, where
$1$ denotes the trivial one-element template.)

\begin{COR}\label{wtc.cor.PT}
  For every finite poset $\calT$ with $T = \{1, \ldots, n\}$ the class $\OPP^\fin_\calT$ has the ordering property.
\end{COR}
\begin{proof}
  Let $\calT$ be a finite poset with $T = \{1, \ldots, n\}$.
  It was shown in~\cite{draganic-masul-mulitposets} that the class $\OPP^\fin_\calT$ has the Ramsey property,
  so Theorem~\ref{wtc.thm.RP-OP} implies that it suffices to show that
  the class $\OPP_\calT$ has the $(W \triangle C)$. But this is straightforward since
  for any $\Sigma \subseteq S_2(\OPP^\fin_\calT)$ we can always take~$\tau = (\{0, 1\}, \Boxed=, \ldots, \Boxed=, \Boxed<)$
  where $0 < 1$.
\end{proof}

Classes of structures having the Ramsey property and the ordering property are particularly intriguing
in the context of Kechris-Pestov-Todor\v cevi\'c correspondence~\cite{KPT}.
Let $G$ be a topological group. Its \emph{action} on $X$ is a mapping $\Boxed \cdot :  G \times X \to X$
such that $1 \cdot x = x$ and $g \cdot (f \cdot x) = (gf) \cdot x$.
We also say that $G$ \emph{acts} on $X$. A \emph{$G$-flow} is a continuous action of a topological group $G$
on a topological space $X$. A \emph{subflow} of a $G$-flow $\Boxed\cdot :  G \times X \to X$
is a continuous map $\Boxed* :  G \times Y \to Y$ where $Y \subseteq X$ is a closed subspace of $X$ and
$g * y = g \cdot y$ for all $g \in G$ and $y \in Y$.
A $G$-flow $G \times X \to X$ is \emph{minimal} if it has no proper closed subflows.
A $G$-flow $u :  G \times X \to X$ is \emph{universal}
if every compact minimal $G$-flow $G \times Z \to Z$ is a factor of~$u$.
It is a well-known fact that for a compact Hausdorff space $X$ there is, up to isomorphism of $G$-flows,
a unique universal minimal $G$-flow, usually denoted by $G \curvearrowright M(G)$.

A topological group $G$ is \emph{extremely amenable}
if every $G$-flow $\Boxed\cdot :  G \times X \to X$
on a compact Hausdroff space $X$ has a fixed point, that is, there is an $x_0 \in X$ such that $g \cdot x_0 = x_0$
for all $g \in G$. Since $\Sym(A)$, the group of all the permutations on a set $A$,
carries naturally the topology of pointwise convergence, permutation groups can
be thought of as topological groups. In \cite{KPT} the authors show the following.

\begin{THM} \cite[Theorem 4.7]{KPT}\label{kpt4.7}
  Let $G$ be a closed subgroup of $\Sym(F)$ for a countable set $F$. Then $G$
  is extremely amenable if and only if $G = \Aut(\calF)$ for a countable ultrahomogeneous structure $\calF$
  whose age has the Ramsey property.
\end{THM}

In case a closed permutation subgroup $G$ on a countable set is not extremely amenable,
\cite{KPT} provides us with a means to compute its universal minimal flow.
Let $\LO(A)$ be the set of all linear orders on $A$ and let $G$ be a closed subgroup of $\Sym(A)$.
The set $\LO(A)$ with the standard product topology is a compact Hausdorff
space and the action of $G$ on $\LO(A)$ given by $x \mathbin{<^g} y \text{ if and only if } g^{-1}(x) < g^{-1}(y)$ is continuous.
This action is usually referred to as the \emph{logical action of $G$ on $\LO(A)$}.

Note that if $\KK^*$ is a reasonable order expansion of $\KK$ and $\KK^*$ has (HP), resp.\ (JEP) or (AP), then
$\KK$ has (HP), resp.\ (JEP) or (AP) \cite{KPT}; consequently if $\KK^*$ is a \Fraisse\ age, then so is $\KK$.
Moreover, assume that $\KK^*$ is a \Fraisse\ age of $\Theta^*$-structures, let $\calF_* = (\calF, \Boxed<)$ be the \Fraisse\ limit
of $\KK^*$ and let $\KK = \restr{\KK^*}\Theta$. Then $\KK^*$ is a reasonable expansion of $\KK$ if
and only if $\KK$ is a \Fraisse\ age and $\calF$ is the \Fraisse\ limit of $\KK$~\cite{KPT}.

\begin{THM} \cite[Theorem 10.8]{KPT}\label{cerp.thm.KPT2}
  Let $\KK^*$ be a \Fraisse\ age which is a reasonable order expansion of a \Fraisse\ age $\KK$.
  Let $\calF$ be the \Fraisse\ limit of $\KK$, let $\calF^* = (\calF, \Boxed\sqsubset)$ be the \Fraisse\ limit of $\KK^*$,
  let $G = \Aut(\calF)$ and $X^* = \overline{G \cdot \Boxed\sqsubset}$ (in the logical action of $G$ on $\LO(F)$).
  Then the logical action of $G$ on $X^*$ is the universal minimal flow of $G$ if and only if
  $\KK^*$ has the Ramsey property and the ordering property.
\end{THM}

We can now present an infinite family of topological groups whose universal minimal flows can be computed using
the Kechris-Pestov-Todor\v cevi\'c correspondence.

\begin{COR}
  Let $\calT$ be a finite poset with $T = \{1, \ldots, n\}$.
  Let $\calP_\calT$ be the \Fraisse\ limit of $\PP^\fin_\calT$, let $\overrightarrow\calP_\calT = (\calP_\calT, \Boxed\sqsubset)$ be the \Fraisse\ limit of
  $\OPP^\fin_\calT$, let $G = \Aut(\calP_\calT)$ and $X^* = \overline{G \cdot \Boxed\sqsubset}$ (in the logical action of $G$ on $\LO(P_\calT)$).
  Then the logical action of $G$ on $X^*$ is the universal minimal flow of $G$.
\end{COR}
\begin{proof}
  It is easy to see that for every template $\calT$ with $T = \{1, \ldots, n\}$ both $\PP_\calT$ and $\OPP_\calT$ are \Fraisse\ ages.
  It was shown in~\cite{draganic-masul-mulitposets} that the class $\OPP^\fin_\calT$ has the Ramsey property, while
  Corollary~\ref{wtc.cor.PT} establishes the ordering property for the class.
  The rest is now an immediate consequence of Theorem~\ref{cerp.thm.KPT2}.
\end{proof}

Let us conclude the paper with another
immediate consequence of Theorems~\ref{cerp.thm.KPT2} and~\ref{wtc.thm.RP-OP}:

\begin{COR}
  Let $\KK^*$ be a \Fraisse\ age satisfying the weak triangle condition
  which is a reasonable order expansion of a \Fraisse\ age $\KK$.
  Let $\calF$ be the \Fraisse\ limit of $\KK$, let $\calF^* = (\calF, \Boxed\sqsubset)$ be the \Fraisse\ limit of $\KK^*$,
  let $G = \Aut(\calF)$ and $X^* = \overline{G \cdot \Boxed\sqsubset}$ (in the logical action of $G$ on $\LO(F)$).
  Then the logical action of $G$ on $X^*$ is the universal minimal flow of $G$ if and only if
  $\KK^*$ has the Ramsey property.
\end{COR}

\section{Acknowledgements}

The author gratefully acknowledges the support of the Ministry of Science, Education and Technological Development of the Republic
of Serbia, Grant No.\ 174019.


\begin{thebibliography}{99}
\bibitem{AH}
  F.\ G.\ Abramson, L.\ A.\ Harrington.
  Models without indiscernibles.
  J.~Symbolic Logic 43 (1978), 572--600.

\bibitem{day-jezek}
  A.\ Day, J.\ Je\v zek.
  The amalgamation property for varieties of lattices.
  Trans.\ Amer.\ Math.\ Soc.\ 286 (1984), 251--256.

\bibitem{draganic-masul-mulitposets}
  N.\ Dragani\'c, D.\ Ma\v sulovi\'c.
  A Ramsey Theorem for Multiposets.
  preprint arXiv:1705.11090

\bibitem{fouche}
  W.\ L.\ Fouch\'e.
  Symmetry and the Ramsey degree of posets.
  15th British Combinatorial Conference (Stirling, 1995). Discrete Math.\ 167/168 (1997), 309--315.

\bibitem{Fraisse1}
  R. Fra\"{\i}ss\'{e}.
  Sur certains relations qui g\'{e}n\'{e}ralisent l'ordre des nombres rationnels.
  C. R. Acad. Sci. Paris 237(1953), 540--542.

\bibitem{Fraisse2}
  R. Fra\"{\i}ss\'{e}.
  Sur l'extension aux relations de quelques propri\'et\'es des ordres.
  Ann. Sci. \'Ecole Norm. Sup. 71(1954), 363--388.

\bibitem{GR}
  R.\ L.\ Graham, B.\ L.\ Rothschild.
  Ramsey's theorem for n-parameter sets.
  Tran.\ Amer.\ Math.\ Soc.\ 159 (1971), 257--292.

\bibitem{gratzer-F}
  G.\ Gr\"atzer.
  Lattice Theory: Foundation.
  Birkh\"auser~2011.

\bibitem{hodges}
  W.\ Hodges. Model Theory.
  Cambridge University Press 1993

\bibitem{Jipsen-Rose}
  P.\ Jipsen, H.\ Rose.
  Varieties of lattices.
  Lecture Notes in Mathematics, Springer 1992.

\bibitem{KPT}
  A.\ S.\ Kechris, V.\ G.\ Pestov, S.\ Todor\v cevi\'c.
  \Fraisse\ limits, Ramsey theory and topological dynamics of automorphism groups.
  GAFA Geometric and Functional Analysis, 15 (2005) 106--189.

\bibitem{kechris-sokic}
  A.\ Kechris, M.\ Soki\'c.
  Dynamical properties of the automorphism groups of the random poset and random distributive lattice.
  Fund.\ Math.\ 218 (2012), 69--94.

\bibitem{masul-rpopos}
  D.\ Ma\v sulovi\'c.
  Pre-adjunctions and the Ramsey property.
	(accepted for publication in European Journal of Combinatorics)

\bibitem{masul-scow}
  D.\ Ma\v sulovi\'c, L.\ Scow.
  Categorical equivalence and the Ramsey property for finite powers of a primal algebra.
  Algebra Universalis 78 (2017), 159--179.


\bibitem{N1995}
  J.\ Ne\v set\v ril.
  Ramsey theory. In: R.\ L.\ Graham, M.\ Gr\"otschel and L.\ Lov\'asz (eds), Handbook of Combinatorics, Vol.~2,
  1331--1403, MIT Press, Cambridge, MA, USA, 1995.

\bibitem{Nesetril}
  J.\ Ne\v set\v ril.
  Ramsey classes and homogeneous structures.
  Combinatorics, probability and computing, 14 (2005) 171--189.

\bibitem{Nesetril-metric}
  J.\ Ne\v set\v ril.
  Metric spaces are Ramsey.
  European Journal of Combinatorics 28 (2007), 457--468.

\bibitem{Nesetril-Rodl-1975}
  J.\ Ne\v set\v ril, V.\ R\"odl.
  Partitions of subgraphs.
  in: M.\ Fiedler (ed), Recent Advances in Graph Theory, 405--412, Academia, Prague, 1975.

\bibitem{Nesetril-Rodl}
  J.\ Ne\v set\v ril, V.\ R\"odl.
  Partitions of finite relational and set systems.
  J.\ Combin.\ Theory Ser.\ A 22 (1977), 289--312.

\bibitem{Nesetril-Rodl-OP-GRA}
  J.\ Ne\v set\v ril, V.\ R\"odl.
  On a probabilistic graph-theoretical method.
  Proc.\ Amer.\ Math.\ Soc.\ 72 (1978), 417--421.

\bibitem{Nesetril-Rodl-lattices}
  J.\ Ne\v set\v ril, V.\ R\"odl.
  Combinatorial partitions of finite posets and lattices -- Ramsey lattices.
  Algebra Universalis 19 (1984), 106--119.

\bibitem{PTW}
  M.\ Paoli, W.\ T.\ Trotter Jr., J.\ W.\ Walker.
  Graphs and orders in Ramsey theory and in dimension theory.
  Graphs and order (Banff, Alta., 1984), 351--394, NATO Adv.\ Sci.\ Inst.\ Ser.\ C Math.\ Phys.\ Sci., 147,
  Reidel, Dordrecht, 1985.

\bibitem{Promel-Voigt}
  H.\ J.\ Pr\"omel, B.\ Voigt.
  Recent results in partition (Ramsey) theory for finite lattices.
  Discr.\ Math.\ 35 (1981), 185--198.

\bibitem{Ramsey}
  F.\ P.\ Ramsey.
  On a problem of formal logic.
  Proc.\ London Math.\ Soc.\ 30 (1930), 264--286.

\bibitem{sokic}
  M.\ Soki\'c.
  Ramsey Properties of Finite Posets.
  Order, 29 (2012) 1--30.

\bibitem{sokic-semilat}
  M.\ Soki\'c.
  Semilattices and the Ramsey property.
  The Journal of Symbolic Logic, 80 (2015), 1236--1259. 
\end{thebibliography}
\end{document}